\documentclass[12pt]{amsart}

\usepackage[a4paper,margin=3.6cm
           ]{geometry} 
\usepackage{amssymb,amsmath,amsthm} 
\usepackage{ifmtarg}            
\usepackage[dvips,all]{xy}      
\usepackage{graphicx}           
\usepackage{psfrag}             
\usepackage{calc}               
\usepackage{color}              
\usepackage[margin=10pt,font=small,labelfont=bf]{caption} 

\usepackage{general}
\usepackage{cantorsets}

\usepackage{hyperref}                       

\newtheorem{thm}{Theorem}[section]
\newtheorem{lem}[thm]{Lemma}
\newtheorem{cor}[thm]{Corollary}

\newtheorem{remark}[thm]{Remark}

\newtheorem{defn}{Definition}[section]

\newtheorem{cnd}{Condition}[section]

\newcommand{\li}{\ar@{-}}   
\newcommand{\ld}{\ar@{--}}  
\newcommand{\ls}{\ar@{.}}   
\newcommand{\gh}{\ar@{}}    
\newcommand{\dm}{\ar@{<->}} 
\newcommand{\iv}{\ar@{|-|}} 

\begin{document}

\title[Differences of Random Cantor Sets ]{ Differences of random Cantor sets and  lower spectral radii }
\author{F.~Michel Dekking \and
Bram Kuijvenhoven}


\maketitle

{\footnotesize Delft University of Technology, Mekelweg 4, 2628 CD Delft, The Netherlands}

\medskip

\begin{center}
{ \bf Abstract}\\
\end{center}
{\footnotesize We investigate the question under which conditions  the algebraic difference between two independent random Cantor sets
$C_1$ and $C_2$  almost surely contains  an interval, and when not. The natural condition is whether the sum $d_1+d_2$ of the Hausdorff dimensions of the sets
 is smaller (no interval) or larger (an interval) than 1. Palis conjectured that \emph{generically} it should be true that
 $d_1+d_2>1$ should imply that $C_1-C_2$ contains an interval. We prove that for 2-adic random Cantor sets generated by a vector of probabilities $(p_0,p_1)$
  the interior of the region where the Palis conjecture does not hold is given by those $p_0,p_1$ which satisfy $p_0+p_1>\sqrt{2}$ and $p_0p_1(1+p_0^2+p_1^2)<1$.
 We furthermore prove a general result which characterizes the interval/no interval property in terms of the lower spectral radius of a set of $2\times 2$
 matrices. }

\section{Introduction}

The \kw{algebraic difference}
 of two  sets $A,B$ of real numbers is defined as:
\begin{align*}
  A - B &:= \{ x - y : x \in A, y \in B \}.
\end{align*}
An interesting situation arises when $A$ and $B$ are relatively small, and $A-B$ large.
For example,  $A$ and $B$ are Cantor sets, but $A-B$ contains an interval.
Whether this will happen or not depends on the size of $A$ and $B$. For instance, if the sum of the Hausdorff dimensions of
$A$ and $B$ is smaller than 1, then  $A-B$ will have a Hausdorff dimension smaller than 1, and can not contain an interval.
A well known conjecture by Palis (\cite{Palis}) states that---conversely--- if
\begin{align} \eqlbl{Palis conjecture cond}
  \Hdim A + \Hdim B > 1,
\end{align}
 then \emph{generically} it should be true that $A-B$ contains an interval.
In this paper we will follow \cite{Larsson} and \cite{DS} and interpret 'generically' as 'almost surely'
with respect to a probability measure.
The central question in this paper is:\\

 \emph{Under which conditions does the algebraic difference between two independent random Cantor sets
 almost surely contain an interval, and when not?}\\

 Here we will consider a canonical class of random Cantor sets, which randomize the classical triadic Cantor set $C$ in
 a natural  way. In this introduction we will give a loose description. Instead of discarding the middle interval and
 keeping the left and the right interval at every step in the construction of $C$  by decreasing intersections of unions of triadic
 intervals, we do the following: fix three numbers $p_0$, $p_1$ and $p_2$ between 0 and 1. Then at every step,
 retain the left interval with probability  $p_0$ (discard it with probability $1\!-\!p_0$), the middle with probability $p_1$
 and the right interval with  probability $p_2$, independently of each other, and of the actions in other intervals at all levels.
 (See also \figref{labeled tree}, where we used trees to describe this recursive construction).

 More generally we consider the $M$-adic case for integers $M=2,3,\dots$, where intervals are recursively divided into $M$ subintervals of equal length,
 which are retained  with \emph{survival probabilities} $p_0, \dots, p_{M-1}$.

 \noindent It turns out that the \emph{cyclic correlation coefficients} $\gamma_k$, defined  by
 \begin{align} \eqlbl{gamma_k def}
 \gam{k} &:= \sum_{i=0}^{M-1} p_ip_{i+k},
\end{align}
 for $k=0,...,\!M-\!1$ play an important role (here the indices $i+k$ should be taken modulo $M$). Indeed, the main result in \cite{DS}
 is the following.

 \begin{thm} \label{thm:DS}
{\rm ([DS08])} Consider two independent random Cantor sets $F_1$ and $F_2$ with survival probablities $p_0, \dots, p_{M-1}$.
\begin{enumerate}
 \item[(a)] If  $\gam{k}>1$ for all $ k=0,...,M\!-\!1$, then $F_1-F_2$ contains an interval a.s.\ on $\set{F_1-F_2\neq\es}$.
 \item[(b)] If  $\gam{k},\gam{k+1}<1$ for some $k$, then $F_1-F_2$ contains no interval a.s.
\end{enumerate}
\end{thm}

 This implies that if $F_1$ and $F_2$ are two independent copies of the random \emph{triadic} Cantor set described above, then their difference $F_1-F_2$
 contains an interval a.s.~ if $p_0p_1+p_1p_2+p_2p_0>1$, and does not contain an interval a.s.~ if $p_0p_1+p_1p_2+p_2p_0<1$. However,
 the triadic case is special, and Theorem~\ref{thm:DS} gives only a partial solution to, e.g., the dyadic case, where intervals
 are split in two all the time, and retained with probability $p_0$ and $p_1$. Here Theorem~\ref{thm:DS} merely yields that $F_1-F_2$
 contains an interval a.s.~ if $2p_0p_1>1$, and does not contain an interval if $p_0^2+p_1^2<1$. In  \secref{2-adic Cantor sets}
 we will fill the gap, and completely classify dyadic random Cantor sets with respect to this property (except on the separating curve).

 The tool that is used in the proof of this result is that of \emph{higher order} Cantor sets, which has been introduced in  \cite{DS}.
 The same tool will enable us in \secref{spectral radius connection} to obtain a general classification result in terms of the \emph{lower spectral radius} of a certain set of matrices.

 A major problem that arises is that the `independent interval' property gets
 lost if one passes to higher order Cantor sets. It is therefore important (and not just for the sake of generalization) to consider a more complex
 mechanism to generate random Cantor sets. The obvious way to allow for dependence is to define a \emph{joint} survival distribution $\mu$ on the set of all subsets of $\{0,...,M\!-\!1\}$. In \secref{Joint survival} we give a version of Theorem~\ref{thm:DS} for this case.

\section{Construction} \seclbl{construction}

The construction of $M$-adic Cantor sets is intimately related to $M$-ary trees and $M$-ary expansions of numbers.

Let $M\ge2$ be an integer.
An $M$-ary tree is a tree in which every node has precisely $M$ children.
The nodes are conveniently identified with strings over an alphabet of size $M$;
 we use the alphabet $\alphabet:=\{0,\dots,M-1\}$.

Strings over $\alphabet$ of length $n$ are denoted as $\tns=\tn$, where $i_1,\dots,i_n\in\alphabet$.
The empty string is denoted by $\es$ and has length $0$.
The concatenation of strings $\tn$ and $\tnj$ is simply denoted by $\tn\tnj$.

The \kw{$M$-ary tree} $\tree$
 is defined as the set of all strings over the alphabet $\alphabet$.
The root node is the empty string $\es$.
The \kw{children} of each node $\tn\in\tree$ are the nodes $\tn i_{n+1}$ for all $i_{n+1}\in\alphabet$.
The \kw{level} of a node corresponds to its length as a string.
For each $n\ge0$, the set of all nodes at level $n$ is denoted by $\tree[n]$.
It thus holds that
\begin{align*}
   \tree
 = \Union_{n\ge0} \tree[n]
 = \Union_{n\ge0} \Union_{i_1\in\alphabet} \cdots \Union_{i_n\in\alphabet} \set{\tn}.
\end{align*}

Strings over the alphabet $\alphabet$ can also be interpreted as $M$-ary expansions of numbers.
For all $\tn\in\tree$ we let $\radix[M]{\tn}$ denote the value of $\tn$ as an $M$-ary number:
\begin{align} \eqlbl{M-ary value}
  \radix[M]{\tn} & := \sum_{k=1}^{n} M^{n-k} i_k.
\end{align}
Consequently, $\radix{\tn}$ takes its value in the range $0,\dots,M^n\!-1$.

\subsection{Random Cantor sets} \seclbl{random Cantor set construction}
We consider the construction of \kw{random $M$-adic Cantor sets} on the interval $[0,1]$.
The construction is an iterative procedure: we start with the entire interval $[0,1]$,
 and at each level of the construction,
 the intervals surviving so far are `split' into $M$ equally sized closed subintervals,
 of which a certain  random subset is allowed to survive at the next level.
The random Cantor set is a stochastic object which consists of those points in $[0,1]$
 that persist at all levels. Here, when we speak of `splitting'  a closed set into some smaller closed sets,
 it should be understood that the smaller sets need not be disjoint, but their interiors are required to be disjoint.

We consider the probability measure $\Cpm$ on the space of $\{0,1\}$-labeled trees $\{0,1\}^\cT$,
 where we label each node $\tn\in\cT$ with $X_\tn\in\{0,1\}$
 and $\mu$ is a probability measure on $2^{2^\alphabet}$
 called the \emph{joint survival measure}. It is of course determined by its restriction to ${2^\alphabet}$,
 which we also denote $\mu$, and call the joint survival \emph{distribution}.
The measure $\Cpm$ is defined by requiring that $\Cprob{X_\es=1}=1$ and that for all $\tn\in\tree$ the  random sets
\begin{align} \eqlbl{X vector sim mu}
  \set{i_{n+1}\in\alphabet: X_{\tn i_{n+1}} = 1} \end{align}
are independent and identically distributed according to $\mu$.

The $n$-th level $M$-adic subintervals of $[0,1]$ are defined by
\begin{align} \eqlbl{M-ary intervals}
 I_\tn & := \sfrac{1}{M^n} \bigl[\radix{\tn}, \radix{\tn} + 1\bigr],
\end{align}
for all $\tn\in\tree$.
The $n$-th level intervals that survive in the $n$-th~level approximation of the random Cantor set
 are the ones that are indexed by the  nodes in the \emph{level $n$ survival} set
\begin{align} \eqlbl{surviving nodes}
  S_n & := \set{\tn: X_{i_1} = X_{i_1i_2} = \dots = X_{\tn} = 1},
\end{align}
for all $n\ge0$.
The random Cantor set $F$ is given as the intersection of all its $n$-th level approximations,
 which we denote by $F^n$:
\begin{align*}
 F := \Intersec_{n=0}^\infty F^n = \Intersec_{n=0}^\infty \Union_{\tn\in S_n} I_\tn.
\end{align*}
An important property of  random Cantor sets is their self-similarity:
 conditional on the survival of any $n$-th level $M$-adic interval,
 the process starting at that interval (scaled by $M^n$) has the same distribution as the whole process,
 which starts at $I_{\es}=[0,1]$.

The vector of \emph{marginal} probabilities $\vc{p}:=\left(p_0,\dots,p_{M-1}\right)$ is defined by
\begin{align} \eqlbl{marginal probabilities}
  p_i := \Cprob{X_i = 1},
\end{align}
for all $i\in\alphabet$.
Note that these marginal probabilities do not need to sum up to $1$.
The joint survival distribution $\mu$ can be chosen such that the $X_i$, $i\in\alphabet$, are
 $M$ independent Bernoulli variables; the respective probabilities of success then
 equal the marginal probabilities $p_0,\dots,p_{M-1}$. In this case we call $F$ an `independent interval' Cantor set.

The traditional deterministic triadic (so $M=3$) Cantor set is obtained
with the measure $\mu$ defined by $\mu(\{0,2\})=1$. Its  vector of marginal probabilities is $\vc{p}=(1,0,1)$.

\begin{figure}[t]
\centering
\includegraphics*[width =12.5cm]{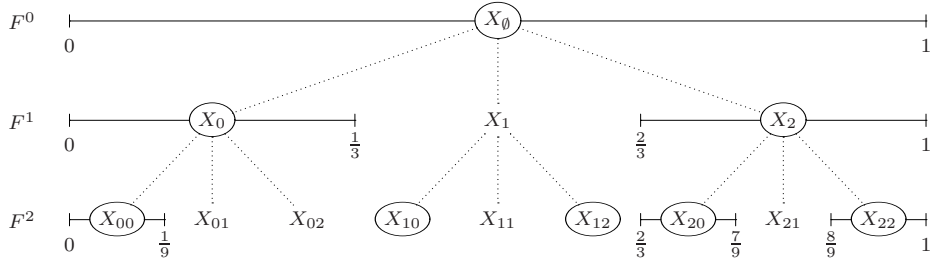}
\caption{The first three levels of a realization of a  tree labeled by $(X_{\tn})$ with $\vc{p}=(1,0,\frac{1}{2})$,
  with the surviving intervals in the approximations $F^n$.
The encircled  $X_{\tn}$ correspond to  nodes obtaining the value $1$.}%
\figlbl{labeled tree}%
\end{figure}

The number $\setcard{S_n}$ of level $n$ intervals selected in $F^n$,
 is a branching process with as offspring distribution the distribution of $\setcard{S_1}$.
Since the $F^n$ are non-increasing, $F=\es$ if and only if
 the branching process $\seq{\setcard{S_n}}$ dies out.
Since $\Cex[\setcard{S_1}]=p_0+\cdots+p_{M-1}=\norm{\vc{p}}_1$,
 it follows that  $F\neq\es$ with positive probability if and only if
 $$ \norm{\vc{p}}_1>1 \quad\text{ or }\quad \Cprob{\setcard{S_1} = 1} = 1.$$
 Discarding the uninteresting case on the right, we will assume henceforth that
\begin{equation} \eqlbl{no extinction}
  \norm{\vc{p}}_1>1.
\end{equation}

\subsection{Algebraic difference}




We consider the algebraic difference $F_1-F_2$ between two independent random $M$-adic Cantor sets $F_1$ and $F_2$.
In general, we denote the joint survival distribution of $F_1$ by $\mu$ and that of $F_2$ by $\lambda$.
The corresponding marginal distributions will be denoted by $\vc{p}$ and $\vc{q}$ respectively.
In   \secref{higher order Cantor sets} and further we will restrict ourselves to the \emph{symmetric case},
 where $\vc{p}=\vc{q}$.
The algebraic difference $F_1-F_2$ can be seen as a projection under 45\degrees{} of the Cartesian product $F_1 \times F_2$.
Thus $F_1-F_2$ is  defined on the product space   of the probability spaces of $F_1$ and $F_2$.
We will use $\bP:=\Cpm\times \Cpl$ to denote the corresponding product measure
 and $\ex$ to denote expectations with respect to this probability.

\section{Triangles and expectations} \seclbl{triangles and expectations}

Let $F_1$ and $F_2$ be two independent $M$-adic random Cantor sets with
 joint survival distributions $\mu$ and $\lambda$, respectively.
Denote by $F_1^n$ and $F_2^n$ their $n$-th level approximations ($n\ge0$) and
 define the following subsets of the unit square $[0,1]^2$:
\begin{align*}
  \Lambda^n &:= F_1^n \times F_2^n, \quad n\ge0, &
  \Lambda   &:= F_1 \times F_2 = \bigcap_{n=0}^\infty \Lambda^n.
\end{align*}
Note that as $F_1^n \downarrow F_1$ and $F_2^n \downarrow F_2$, also $\Lambda^n \downarrow \Lambda$.
Let $\phi:\reals^2\to\reals$ denote the 45\degrees{} projection given by $\phi(x,y)=x-y$, then $F_1-F_2=\phi(\Lambda)$.
As $\phi$ is a continuous function and $\{\Lambda^n\}_{n=0}^\infty$ is
 a non-increasing sequence of compact  sets,
 it follows that the algebraic difference $F_1-F_2$ can be written as
\begin{align*}
    F_1\!-\!F_2
  = \phi(\Lambda)
 &= \phi\Big(\!\bigcap_{n=0}^\infty \Lambda^n\!\Big)
  = \bigcap_{n=0}^\infty \phi(\Lambda^n)
  = \bigcap_{n=0}^\infty \phi(F_1^n \!\times\!F_2^n)
  = \bigcap_{n=0}^\infty (F_1^n\! -\! F_2^n).
\end{align*}

\subsection{Squares, columns and triangles}

\begin{figure}[t]
\centering
\includegraphics*[width =12cm]{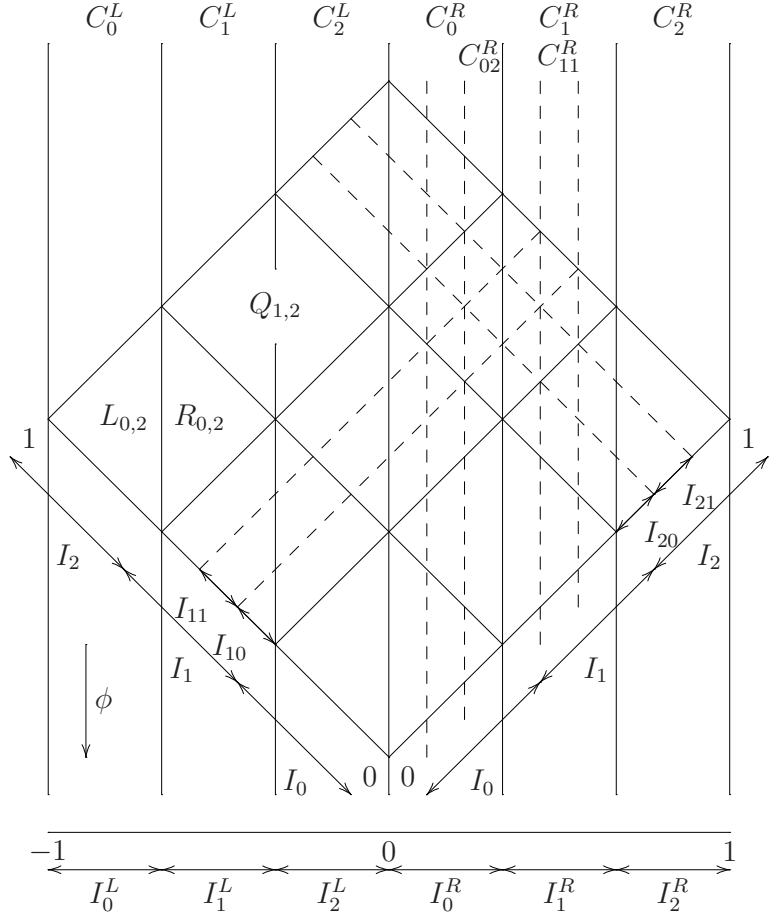}
\caption{
 An illustration for $M=3$ of the unit square $[0,1]^2$ rotated by 45\degrees{},
  being projected by $\phi$ to a $\sqrt{2}$-scaled-down version of $[-1,1]$.
 The columns $C^U_{\tnks}$ split the $n$-th level squares $Q_{\tns,\tnjs}=I_{\tns} \times I_{\tnjs}$ into
  the `left' and `right' triangles $L_{\tns,\tnjs}$ and $R_{\tns,\tnjs}$.}%
\figlbl{projection}%
\end{figure}

The $\Lambda^n$ are unions of $M$-adic squares
\begin{align*}
  Q_{\tn,\tnj} := I_{\tn} \times I_{\tnj},
\end{align*}
with $\tn,\tnj\in\tree[n]$ and $n\ge0$.
See \figref{projection} for a graphical representation of these $M$-adic squares and their $\phi$-projections.
Observe that the projections $\phi(Q_{\tn,\tnj})$ are equal to
  unions of two subsequent level $n$ $M$-adic intervals in $[-1,1]$.

In order to be able to represent the $M$-adic intervals that are in $[-1,0]$,
 we generalize our notation of $M$-adic intervals on $[0,1]$ to the entire real line.
For any $n\ge0$ and $k\in\integers$ we define
\begin{align} \eqlbl{I_n(k) def}
  I_n(k) := \tfrac{1}{M^n} \left[k,k+1\right] = \left[\tfrac{k}{M^n},\tfrac{k+1}{M^n}\right].
\end{align}
Note that $I_n(k+M^ni) = I_{\tnk} + i$ for all $k=\radix[M]{\tnk}$ and $i\in\integers$.
See also \figref{interval enumeration}.
The inverse images of these intervals under $\phi$ form diagonal `columns' in the plane $\reals^2$, denoted by
\begin{align*}
 C_n(k) &:= \inv{\phi}\left(I_n(k)\right),
\end{align*}
for all $k\in\integers$.

When rotating the unit square $[0,1]^2$ by 45\degrees{}, as in \figref{projection},
  the columns with $\phi$-image in $[-1,0]$ intersect with the `left' half of the unit square
 and those with $\phi$-image in $[0,1]$ intersect with the `right' half.
For this reason we distinguish between `left' and `right' $M$-adic intervals and columns
 by defining for any $\tnk\in\tree$
\begin{align*}
\begin{aligned}
&I^L_\tnk \!:=\!I_\tnk\!-\!1=\!I_n(\radix{\tnk}\!-\!M^n),\, I^R_\tnk\! :=\!I_\tnk=\! I_n(\radix{\tnk}), \\
&C^L_\tnk := C_n(\radix{\tnk}-M^n),\; C^R_\tnk := C_n(\radix{\tnk}).
\end{aligned}
\end{align*}
In fact, any $n$-th level $M$-adic square $Q_{\tn,\tnj}$
 is split into a `left' and a `right' triangle by the $M$-adic columns.
These triangles are called $L$-triangles and $R$-triangles, and are denoted by
\begin{align} \eqlbl{triangles}
\begin{aligned}
  L_{\tn,\tnj} &:= Q_{\tn, \tnj} \intersec C_n\left(\radix{\tn}-\radix{\tnj}-1\right), \\
  R_{\tn,\tnj} &:= Q_{\tn, \tnj} \intersec C_n\left(\radix{\tn}-\radix{\tnj}  \right),
\end{aligned}
\end{align}
for any $\tn,\tnj\in\tree$.

\begin{figure}[t]
\centering
\includegraphics*[width =12.5cm]{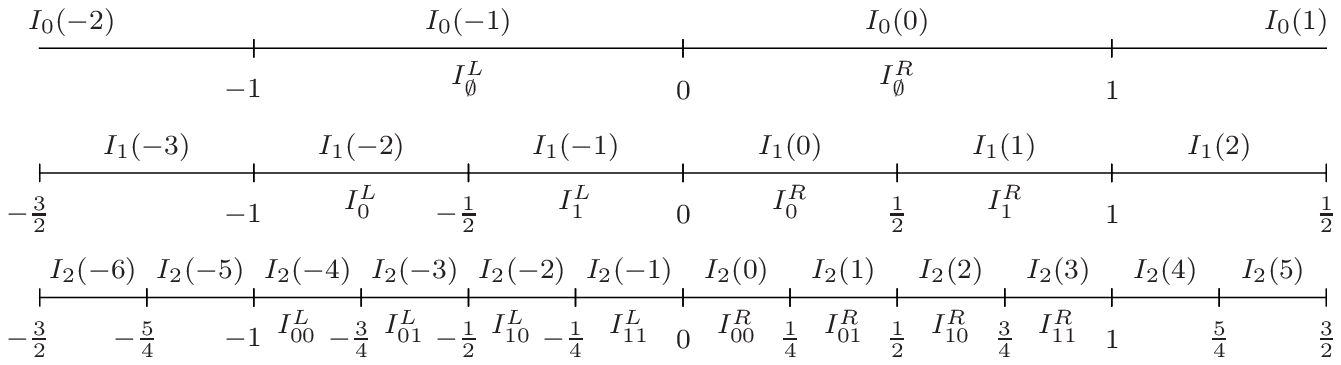}
\caption{Two ways of enumerating $M$-adic intervals (with $M=2$) --- $I_n(i)$, $i\in\integers$,
   and  $I^U_{\tnks}$, with $U\in\set{L,R}$ and $\tnks\in\tree$. Here $n=0,1,2$.}
\figlbl{interval enumeration}
\end{figure}

\subsection{Triangle counts}


For all $U,V\in\set{L,R}$ and $\tnks\in\tree$ we let
\begin{align*}
    Z^{UV}(\tnks)
 := \setcard{\set{\big(\tns,\tnjs\big): Q_{\tns,\tnjs}\subseteq\Lambda^n, V_{\tns,\tnjs}\subseteq C^U_{\tnks}}}
\end{align*}
denote the number of level~$n$ $V$-triangles in $\Lambda^n \intersec C^U_\tnks$.
Note that these $V$-triangles have been generated by the level~$0$ $U$-triangle.
We also denote the total number of $V$-triangles in columns $C^L_{\tnks}$ and $C^R_{\tnks}$ together by
\begin{align*}
  Z^V(\tnks) := Z^{LV}(\tnks) + Z^{RV}(\tnks),
\end{align*}
for all $\tnks\in\tree$.
These triangle counts (and their self-similarity property) are illustrated in \figref{dual columns}.

 \begin{figure}[b]
\centering
\includegraphics*[width =11cm]{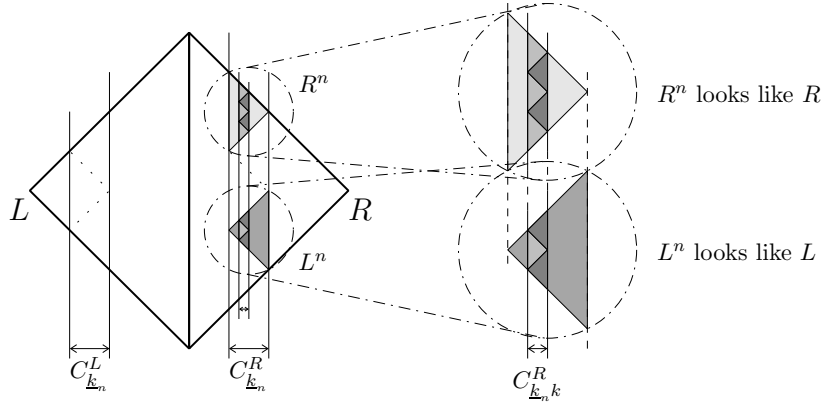}
\caption{A comparison of a level $n$ left triangle $L^n$ and right triangle $R^n$
 with the level $0$ left triangle $L$ and right triangle $R$.
 In particular, the expected number of (level $n+1$) $V$ triangles
 in the intersection of the subcolumn $C^R_{\tnks k}$ with the triangles $L^n$ and $R^n$ together
 equals $\ex Z^V(k)$, the expected number of (level $1$) $V$ triangles in $C^L_k$ and $C^R_k$ together.}%
\figlbl{dual columns}%
\end{figure}

An important observation is that an $M$-adic interval $I^U_{\tnks}$ is absent in
  $\phi(\Lambda^n)$ exactly when
 there are no triangles  in the corresponding column $C^U_{\tnks}$ in $\Lambda^n$:
\begin{align} \eqlbl{interval not in projection iff triangle counts zero}
  I^U_\tnks \not\subseteq \phi(\Lambda^n) \iff Z^{UL}(\tnks)=Z^{UR}(\tnks)=0.
\end{align}

The triangle counts $Z^{UV}(\tnks)$, with $k_1,k_2,\dots$ a fixed path,
 constitute a two type branching process in a varying environment with interaction:
  the interaction comes from the dependency between triangles that are \emph{aligned},
 i.e., triangles contained in respective squares $Q_{\tn,\tnj}$ and $Q_{i_1'\dots i_n',j_1'\dots j_n'}$
 with $\tn=i_1'\dots i_n'$ or $\tnj=j_1'\dots j_n'$.
 The \kw{expectation matrices} of the two type branching process are given by:
\begin{align} \eqlbl{exmat def}
 \exmat{\tnks} := \begin{bmatrix}
   \ex Z^{LL}(\tnks) & \ex Z^{LR}(\tnks) \\
   \ex Z^{RL}(\tnks) & \ex Z^{RR}(\tnks)
 \end{bmatrix},
\end{align}
where $\tnks\in\tree$.
These matrices satisfy the basic relation
\begin{equation} \eqlbl{exmat expansion}
 \exmat{\tnk} = \exmat{k_1} \cdots \exmat{k_n},
\end{equation}
for all $\tnk\in\tree$.


\subsection{Correlation coefficients} \seclbl{correlation coefficients}
Define the \kw{cyclic cross-correlation coefficients}
\begin{align} \eqlbl{gamma_k def}
 \gam{k} &:= \sum_{i=0}^{M-1} q_ip_{i+k},
\end{align}
 where the indices of $p$ should be taken modulo $M$, and $k\in\alphabet$.

This definition is extended to $\integers$ by setting
 $\gam{k+iM}:=\gam{k}$ for all $i\in\integers$.
For the symmetric case $\vc{p}=\vc{q}$, these coefficients are called the
 \kw{cyclic auto-correlation coefficients}.
For brevity, however, we will use the shorter term \kw{correlation coefficients}.
The smallest correlation coefficient value is denoted by
\begin{align} \eqlbl{gamma def}
 \gamma &:= \min_{k\in\alphabet} \gam{k}.
\end{align}
\lemref{exmat sum to gamma} below motivates the definition of the correlation coefficients,
 as they are in fact the triangle count expectations $\ex{Z^V(k)}$, $V\in\{L,R\}$.

For convenience, we define the vector $\vc{e}:=\evec$.

\begin{lem} \lemlbl{exmat sum to gamma}(\cite{DS})
For all $k\in\alphabet$ we have
\begin{align*}
  \vc{e}\exmat{k} = \evec\exmat{k} = \lvec{\ex Z^L(k)}{\ex Z^R(k)} = \gamvec{k}.
\end{align*}
\end{lem}

\begin{proof}
As in \cite{DS} this follows from \eqref{eq:triangles} with $n=1$,
$$ \prob{Q_{i,j} \subseteq \Lambda^1}
  = \prob{I_i \subseteq F^1_1, I_j \subseteq F^1_2}
  = \Cprob{I_i \subseteq F^1_1}\Cprol{I_j \subseteq F^1_2}
  = p_iq_j $$
 and some careful bookkeeping. 
\end{proof}


A property that in general holds \emph{only} for the symmetric case (i.e., $\vc{p}=\vc{q}$),
 is that $\gam{0}$ is the largest of the auto-correlation coefficients:
\begin{equation}  \eqlbl{CS}
 0\le\gamma_k\le \gamma_0.
 \end{equation}
This follows easily with the Cauchy-Schwarz inequality.
\section{The basic result with joint survival distributions} \seclbl{Joint survival}

In this section we generalize Theorem~\ref{thm:DS} of \cite{DS} to joint survival distributions, \emph{and}
to the asymmetric case.

In the setting of general joint survival distributions $\mu$ and $\lambda$,
 the condition stated below is sufficient for the theorem to hold.

For a joint survival distribution $\mu:2^\alphabet\to[0,1]$
 we define its \kw{marginal support}:
\begin{align} \eqlbl{msupp def}
 \msupp{\mu} := \Union \set{S\subseteq\alphabet:\mu(S)>0}.
\end{align}
In other words, the marginal support is the set of $i\in\alphabet$ for which it holds that $p_i=\Cprob{X_i = 1}>0$.
For example, take $M=4$ and $\mu$ defined by $\mu(\{0,3\})=\mu(\{1,3\})=\mu(\{3\})=\frac{1}{3}$,
 then the marginal support is $\msupp{\mu}=\{0,1,3\}$.

\begin{cnd} \cndlbl{joint event condition}
A joint survival distribution $\mu:2^\alphabet\to[0,1]$ satisfies the \kw{joint survival condition (JSC)}
 if it assigns a positive probability to its marginal support: $\mu(\msupp{\mu})>0$.
\end{cnd}

In the `independent interval' case $\mu$
 satisfies the joint survival condition since in that case $\mu(\msupp{\mu})=\prod_{i\in\msupp{\mu}} p_i>0$.
In the example above, where $\msupp{\mu}=\{0,1,3\}$, the JSC is not satisfied.

\begin{thm} \label{thm:central}
Consider two independent random Cantor sets $F_1$ and $F_2$ whose
 joint survival distributions satisfy \cndref{joint event condition}, the joint survival condition.
\begin{enumerate}
 \item[(a)] If  $\gam{k}>1$ for all $ k\in\alphabet$, then $F_1-F_2$ contains an interval a.s.\ on $\set{F_1-F_2\neq\es}$.
 \item[(b)] If  $\gam{k},\gam{k+1}<1$ for some $ k\in\alphabet$, then $F_1-F_2$ contains no interval a.s.
\end{enumerate}
\end{thm}
\begin{proof}
A proof for the symmetric case and `independent interval' Cantor sets is given in \cite{DS}.
An extension of the  proof to the asymmetric case and general survival distributions
 satisfying the joint survival condition is easy for part (b) (where, moreover, the JSC is not needed), but for part (a) several complications arise.
 We will discuss these in \secref{general proof}.
\end{proof}


The joint survival condition is not a necessary condition. Here is an example where the JSC does not hold, but where
$F_1-F_2$ contains an interval: take $M=5$, and $\mu(\{0,1,2,3\})=p=1-\mu(\{1,2,3,4\})$ where $p$ is arbitrary between 0 and 1.
Then $\gamma=3$ for all $p$, but the JSC does not hold. However,  any  realisation $\Lambda^n$ contains the \emph{deterministic}
set $\tilde{\Lambda}^n$ generated by $\tilde{\mu}$ defined by $\tilde{\mu}(\{1,2,3\})=1$. A simple geometric analysis shows that
$\tilde{F}=[-1/2,1/2]$, and hence $F$ must contain this interval. (An algebraic alternative is to use Theorem 2 of \cite{DS}:
the collection of reduced matrices of $\tilde{\mu}$ is $\{T_3,T_9,T_{12}\}$, and since this set is closed under (mutual)
multiplications, this theorem tells us that $\tilde{F}$ will contain an interval.)


\section{Higher order Cantor sets} \seclbl{higher order Cantor sets}

From now on  we restrict ourselves to the symmetric case,
 that is the vectors of marginal probabilities satisfy $\vc{p}=\vc{q}$.
Whether similar results can be obtained for $\vc{p}\neq\vc{q}$ is yet unknown.

Essentially, the \kw{$n$-th order random Cantor set} is constructed by `collapsing' $n$ steps
 of its construction into one step.
We denote all entities of an $n$-th order random Cantor set with a superscript $^{(n)}$.

The alphabet $\alphabet^{(n)}$ of the $n$-th order random Cantor set is  $\{0,\dots,M^n\!-1\}$,
with elements $i^{(n)}=\radix{\tn}$. To reduce the overloaded notation we will omit the superscript here, and simply write
$i=i^{(n)}$.

The joint survival distribution $\mu^{(n)}:2^{\alphabet^{(n)}}\to[0,1]$, which by definition is the distribution of the sets
\begin{align*}
  \set{i_{m+1}\in\alphabet^{(n)}:X^{(n)}_{\tm i_{m+1}}=1}
\end{align*}
for all $\tm\in\tree^{(n)}$,
 is determined uniquely by requiring that
\begin{align*}
  X^{(n)}_i \sim \prod_{d=1}^n X_{i_1\dots i_d}
               = X_{i_1} X_{i_1i_2} \cdots X_{i_1\dots i_n},
\end{align*}
for all $i=\radix{\tn}\in\alphabet^{(n)}$, where the $X_{i_1\dots i_d}$ are defined in \eqref{eq:X vector sim mu}.

It is clear that the higher order marginal probabilities are given by
\begin{align}
  p_i^{(n)} := \prob[\bP_{\mu^{(n)}}]{X^{(n)}_i=1}
             = \prod_{d=1}^n \prob[\bP_\mu]{X_{i_1\dots i_d}=1}
             = \prod_{d=1}^n p_{i_d},
\end{align}
for all $i=\radix{\tn}\in\alphabet^{(n)}$.

\subsection{Joint survival.}
Note that when $\mu$ describes an  `independent interval' Cantor set,
 it does \emph{not} hold in general that $\mu^{(n)}$ corresponds to an  `independent interval' Cantor set with marginals $p^{(n)}_i$.
This is because e.g.\ the products $X_{0}X_{00}$ and $X_{0}X_{01}$ are not independent: they share the $X_0$ term.

However, the joint survival condition (\cndref{joint event condition}) nicely propagates to higher order Cantor sets.
Suppose $\mu$ satisfies the JSC. We have
\begin{align*}
  \msupp{\mu^{(n)}} &= \Big\{i=\radix{\tn}\in\alphabet^{(n)}: p^{(n)}_i = \prod_{d=1}^n p_{i_d} > 0\Big\},
\end{align*}
which implies that $\mu^{(n)}$ satisfies the joint survival condition as well:
\begin{align*}
     \mu^{(n)}\left(\msupp{\mu^{(n)}}\right)
 & = \left(\mu\left(\msupp{\mu}\right)\right)^{1+a+a^2+\dots+a^{n-1}} > 0,
\end{align*}
where $a:=\setcard{\msupp{\mu}}$ is the cardinality of the marginal support of $\mu$.


The key observation regarding higher order Cantor sets is that for all $n\ge1$
\begin{align}
  F^{(n)} \sim \Intersec_{m=1}^\infty F^{n\cdot m}  = F,
\end{align}
 hence statements such as Theorem \ref{thm:central}
 can be applied to higher order correlation coefficients $\gamn{k}$ in order to get results
 not only for $\smash{F_1^{(n)}-F_2^{(n)}}$, but for $F_1-F_2$ as well.

\subsection{Expectation matrices}
The expectation matrices of the higher order Cantor sets satisfy the following factorization property:
 for all $\tnk\in\tree$:
\begin{align} \eqlbl{ho exmat expansion}
 \exmat[(n)]{\radix[M]{\tnk}}
  = \exmat{\tnk}
  = \exmat{k_1} \cdots \exmat{k_n}.
\end{align}


For the $\gamnk$ we can use \lemref{exmat sum to gamma},
 which relates the $\gamnk$ to the expectation matrices.
Recall that $\vc{e}:=\evec$.
For all $\tnk\in\tree$:
\begin{align} \eqlbl{gamma as matrix product}
    \gamvec[(n)]{k}
  = \vc{e} \exmat[(n)]{k}
  = \vc{e} \exmat{k_1} \cdots \exmat{k_n},
\end{align}
where $k=\radix[M]{\tnk}$.
We define $ \gam[(0)]{k}= 1$ for all $k\in\integers$.
A direct consequence of \eqref{eq:gamma as matrix product} is the following recursive relation between
 $\gamnk$ of different orders, written as a matrix multiplication:
\begin{align} \eqlbl{gamma recursion}
  \gamvec[(n)]{M^{n-m}k+l} = \gamvec[(m)]{k} \exmat[(n-m)]{l},
\end{align}
for all $0\le m \le n$, $k\in\alphabet^{(m)}$ and $l\in\alphabet^{(n-m)}$.
In fact we can take $k\in\integers$ here because, by definition,
 $\gam[(m)]{k+M^mi}=\gam[(m)]{k}$ for all $i\in\integers$, $m\ge0$ and $k\in\integers$.
Note that in particular
\begin{align} \eqlbl{one step gamma recursion}
  \gamvec[(n+1)]{Mk+l} = \gamvec[(n)]{k} \exmat{l},
\end{align}
for all $n\ge0$, $k \in \integers$ and $l\in\alphabet$.



\subsection{An alternative notation} \seclbl{an alternative notation}
In \eqref{eq:exmat def} the expectations $\ex Z^{UV}(\tnks)$ are put into the matrices $\exmat{\tnks}$.
In this section is discussed another way to alias these expectations
 in such a way that equivalent entries get the same alias.
This aliasing scheme also provides an alternative representation of recursion relation \eqref{eq:one step gamma recursion}.

The expectation matrices $\exmat{k}$ together contain $4M$ entries, but not all entries are distinct.
The number of left triangles in column $k$ for example equals the number of right triangles in column $k+1$.
The reason is simple: the left triangles in column $k$ and the right triangles in column $k+1$
 are the respective halves of the same $M$-adic squares --- squares that have the same projection under $\phi$.
The following equations point out which triangle count expectations are always equal to each other,
 and define the aliasing scheme $m_e$, $e\in\{-M,\dots,M\}$, for those $e$ satisfying $\abs{e}<M$:
\begin{align*}
  \begin{alignedat}{3}
    m_{k-M} &:= \ex Z^{LL}(k-1) &&= \ex Z^{LR}(k),  &&\quad 0<k<M, \\
    m_0     &:= \ex Z^{LL}(M-1) &&= \ex Z^{RR}(0),  &&\\
    m_k     &:= \ex Z^{RL}(k-1) &&= \ex Z^{RR}(k),  &&\quad 0<k<M.
  \end{alignedat}
\end{align*}
The remaining two aliases, $m_e$ with $\abs{e}=M$, are given by
\begin{align*}
 m_{-M} &:= \ex Z^{LR}(0)   = 0, &
 m_M    &:= \ex Z^{RL}(M-1) = 0.
\end{align*}
These expectations are always zero because there are no right triangles in the leftmost column
 and no left triangles in the rightmost column.

The  matrices $\exmat{k}$ take their entries from these $m_e$ as follows:
\begin{align} \eqlbl{exmats with m_e entries}
 \exmat{k} = \lmat{m_{k+1-M}}{m_{k-M}}{m_{k+1}}{m_k}
\end{align}
for $k\in\{0,\dots,M-1\}$.
Consequently, recursion relation \eqref{eq:one step gamma recursion} can be written using the $m_e$ as
\begin{align} \eqlbl{one step gamma recursion with m_e}
  \gam[(n+1)]{Mk+l} = m_{l-M} \gamn{k+1} + m_l \gamn{k}
\end{align}
for any $n\ge0$, $k\in\integers$ and $l\in\{0,\dots,M\}$.
Note in particular that the case $l=M$ is allowed and valid here;
 it corresponds to the \emph{left} column of \eqref{eq:one step gamma recursion} with $l=M-1$ there.

Relation \eqref{eq:gamma recursion} can also be written in terms of the $m_e$,
 though the higher order version $m^{(n-m)}_e$ is required
 as we are dealing with entries from $\exmat[(n-m)]{l}$ instead of $\exmat{l}$.
The equation that relates the $n$-th order to the $m$-th order is:
\begin{align} \eqlbl{gamma recursion with m_e}
  \gamn{M^{n-m}k+l} &= m^{(n-m)}_{l-M^{n-m}} \gam[(m)]{k+1} + m^{(n-m)}_l \gam[(m)]{k},
\end{align}
for all $0\le m\le n$, $k\in\integers$ and $l\in\{0,\dots,M^{n-m}\}$.

An explicit formula for the $m_e$ can also be given; some careful bookkeeping yields the following set of formulas:
\begin{align} \eqlbl{explicit m_e formula}
 m_e = \sum_{i,j\in\alphabet: i-j = e} p_ip_j
     = \sum_{i=\max(0,e)}^{\min(M,M+e)-1} p_ip_{i-e}
     = \sum_{j=\max(0,-e)}^{\min(M,M-e)-1} p_{j+e}p_j,
\end{align}
for $e\in\{-M,\dots,M\}$.
For $\abs{e}=M$, these are sums over empty sets, which by convention evaluate to $0$.

From equation \eqref{eq:explicit m_e formula} it is immediately clear that
 $m_e=m_{-e}$ for all $e\in\{-M,\dots,M\}$, so $m_0, \dots, m_M$ is all we need to encode the matrices.


\subsection{Bounded skewness}

The scope of \thmref{central} can be extended if eventually for some $n$ all \hogams{} are strictly greater than 1,
 or when two subsequent ones, say $\gamn{k}$ and $\gamn{k+1}$, are strictly less than 1.
If it were possible that the ratio between two successive \hogams{} could become arbitrarily large as $n\to\infty$,
 it could be hard to show that we would eventually end up in one of these cases;
 perhaps one \hogam{} would always remain below $1$, while all others were already above $1$.

The next lemma shows that the scenario of unbounded neighbor ratio
 is not the case if all $m_e$ are positive for $\abs{e}<M$.
In the next section this lemma will play a key role in proving that
 --- except for a (topologically) inconsiderable set of marginal probabilities $\vc{p}$ ---
 \thmref{central} can always be fruitfully applied to higher order Cantor sets for orders $n$ that are large enough.

Suppose that $m_e>0$ for $\abs{e}<M$, then
\begin{align} \eqlbl{R def}
  R := \min_{\substack{0<k<M \\ k'\in\{k-1,k+1\}}} \frac{\min(m_{k-M},m_k)}{\max(m_{k'-M},m_{k'})}
\end{align}
is well defined, and $R\in(0,1]$.

\begin{lem} \lemlbl{bounded skewness}
If $m_e>0$ for $e=1,\dots,M-1$, then
\begin{align} \eqlbl{bounded skewness}
     \frac{\min(\gamn{k},\gamn{k+1})}{
           \max(\gamn{k},\gamn{k+1})} \ge R,
\end{align}
for all $n\ge0$ and $k\in\integers$, where $R$ is defined as in \eqref{eq:R def}.
\end{lem}
\begin{proof}
Assume that $m_e>0$ for all $\abs{e}<M$.
Note that by applying a simple inductive argument over $n$
 to recursion relation \eqref{eq:one step gamma recursion with m_e}
 with initial conditions $ \gam[(0)]{k}= 1$ for all $k\in\integers$
 --- and using that for all $l\in\{0,\dots,M\}$ it holds that $m_{l-M}$ and $m_l$ are not both equal to zero ---
 it follows that $\gamn{k}>0$ for all $n\ge0$ and $k\in\integers$.
Hence the fraction in \eqref{eq:bounded skewness} is well defined.

An equivalent formulation of \eqref{eq:bounded skewness} is that
 for all $n\ge0$ and all neighbor indices $k,k'\in\integers$ --- i.e., $\abs{k-k'}=1$ ---
 it holds that $\gamn{k} / \gamn{k'} \ge R$.
We will prove this version, as it turns out to be slightly more convenient to work with.

We will use induction over $n$.
For $n=0$ it holds that $\gam[(0)]{k} / \gam[(0)]{k'} = 1/1 = 1 \ge R$.
Now assume that the lemma holds for an $n\ge0$.
Let $k,k'\in\integers$ with $\abs{k-k'}=1$ be arbitrary.

First assume that $M$ does \emph{not} divide $k$,
 then we can write $k=Mk_1+k_2$ for some $k_1\in\integers$ and some $0<k_2<M$.
If we define $k'_2$ by setting $k'_2-k_2=k'-k$, then $k'=Mk_1+k'_2$ with $0\le k'_2\le M$.
Using recursion relation \eqref{eq:one step gamma recursion with m_e}, we obtain:
\begin{align*}
     \frac{\gam[(n+1)]{k}}{\gam[(n+1)]{k'}}
   = \frac{\gam[(n+1)]{Mk_1+k_2}}{\gam[(n+1)]{Mk_1+k'_2}}
  &= \frac{m_{k_2-M}\gamn{k_1+1} + m_{k_2}\gamn{k_1}}{
           m_{k'_2-M}\gamn{k_1+1} + m_{k'_2}\gamn{k_1}} \\
&\ge \frac{\min(m_{k_2-M},m_{k_2}) \left(\gamn{k_1+1} + \gamn{k_1}\right)}{
           \max(m_{k'_2-M},m_{k'_2}) \left(\gamn{k_1+1} + \gamn{k_1}\right)}
 \ge R.
\end{align*}
Note how the fact that $0<k_2<M$ is used in the last inequality.
Also, we didn't need to use the induction hypothesis for this case.

Now assume that $M$ \emph{does} divide $k$, and write $k=Mk_1$ for some $k_1\in\integers$.
Because $m_e+m_{e-M}=\gam{e}\le\gam{0}=m_0$ for all $e\in\{0,\dots,M\}$
 --- here we use equation \eqref{eq:CS} ---
 we have the bounds
\begin{align*}
  \begin{alignedat}{6}
	    \gam[(n+1)]{k+1}
    & = \gam[(n+1)]{Mk_1+1}
   && = m_{1-M} && \gamn{k_1+1} &&+ m_1     &&\gamn{k_1}
  &&\le m_0 \max(\gamn{k_1}, \gamn{k_1+1}), \\
	    \gam[(n+1)]{k-1}
    & = \gam[(n+1)]{M(k_1-1)+(M-1)}
   && = m_{-1}  && \gamn{k_1}   &&+ m_{M-1} &&\gamn{k_1-1}
  &&\le m_0 \max(\gamn{k_1}, \gamn{k_1-1}).
  \end{alignedat}
\end{align*}
If we define $k'_1$ by setting $k'_1-k_1=k'-k$, $k'_1$ and $k_1$ are neighbours,  we can use these bounds as follows:
\begin{align*}
     \frac{\gam[(n+1)]{k}}{
           \gam[(n+1)]{k'}}
   = \frac{\gam[(n+1)]{Mk_1}}{
           \gam[(n+1)]{k'}}
&\ge \frac{m_0\gamn{k_1}}{
           m_0\max(\gamn{k_1}, \gamn{k'_1})}
   = \min\left(1,\frac{\gamn{k_1}}{
                       \gamn{k'_1}}\right)
 \ge R,
\end{align*}
where in the last step we used the induction hypothesis.
\end{proof}

\section{The lower spectral radius connection} \seclbl{spectral radius connection}

In the previous section higher order Cantor sets were introduced.
The corresponding higher order expectation matrices $\exmat[(n)]{k}$ are
 products of the first-order matrices $\exmat{k}$, as \eqref{eq:gamma as matrix product} shows.
It will turn out that,
 for most $\vc{p}$ the lower spectral radius of the set of matrices $\exmat{k}$, $k\in\alphabet$,
 captures exactly the information needed to determine
 whether $F_1-F_2$ contains an interval or not.
The following definition generalizes the concept of `spectral radius' to a \emph{set} of matrices.

\begin{defn} \label{defn:lower and upper spectral radius}(\cite{Gurvits})
Let $\normop$ be a submultiplicative norm on $\reals^{d\times d}$
 and $\Sigma\subseteq\reals^{d\times d}$ a finite non-empty set of square matrices.
The \kw{lower spectral radius}  of $\Sigma$ is defined by
\begin{equation} \eqlbl{lower spectral radius defn}
      \lsr{\Sigma}  := \liminf_{n\to\infty} \lsrn{\Sigma,\normop}, \quad
      \lsrn{\Sigma,\normop}  := \min_{A_1,\dots,A_n\in\Sigma} \norm{A_1\cdots A_n}^{1/n}.
\end{equation}
\end{defn}

It is easily seen that the definition of the lower  spectral radius is independent
 of the particular choice of matrix norm.

\subsection{A spectral radius characterization}
For the algebraic difference between two $M$-adic random Cantor sets,
 we are interested in the lower spectral radius of the set of matrices
\begin{align} \eqlbl{Sigma M}
     \Sigma_{\cM}:=\{\cM(0),\dots,\cM(M-1)\}.
\end{align}

The following theorem provides a generalization of \thmref{central} for the symmetric case
 using the concept of the lower spectral radius.

 First we need yet another notion.

 \noindent We call the collection of matrices $\{\exmat{0}, \cdots, \exmat{M-1}\}$ \emph{irreducible} if
 $$\ex Z^{LR}(k)>0,\qquad k=1,\dots,M-1,$$
  and
  $$\ex Z^{RL}(k)>0,\qquad k=0,\dots,M-2.$$
 Note that the constraints on $k$ are natural, because \emph{always} $\ex Z^{LR}(0)=0$ and $\ex Z^{RL}(M-1)=0$.

%
\begin{thm} \thmlbl{central with lsr}
Consider the algebraic difference $F_1-F_2$ between two $M$-adic independent random Cantor sets $F_1$ and $F_2$
  whose joint survival distributions satisfy the joint survival condition, have equal marginal probabilities and lead
 to an irreducible collection  $\Sigma_\cM$ as in \eqref{eq:Sigma M}.
\begin{itemize}
 \item[(a)] If $\lsr{\Sigma_\cM}>1$, then $F_1-F_2$ contains an interval a.s.~on $\set{F_1\!-\!F_2\neq\es}$.
 \item[(b)] If $\lsr{\Sigma_\cM}<1$, then $F_1\!-\!F_2$ contains no intervals a.s.
\end{itemize}
\end{thm}
\begin{proof}
\newcommand{\matrixset}{\Sigma_\cM}
Let $\onenormop$  denote the maximum absolute column sum norm.

First assume that $\lsr{\matrixset}<1$.
Then there exists a number $n$ such that $\lsrn{\matrixset,\onenormop}<1$.
In particular, there exist $k_1,\dots,k_n\in\alphabet$ such that
\begin{align*}
  \onenorm{\cM(k_1)\cdots\cM(k_n)} < 1.
\end{align*}
It follows that for  $k=\radix{k_1\dots k_n}\in\alphabet^{(n)}$
\begin{align} \eqlbl{max gamma product norm inequality}
     \max\left(\gamma^{(n)}_{k+1}, \gamma^{(n)}_k\right)
     = \onenorm{\cM^{(n)}_k}=\onenorm{{\cM}(k_1)\cdots{\cM}(k_n)} < 1.
\end{align}
From \thmref{central}  statement  (b) follows.

Now assume that $\lsr{\matrixset}\ge 1+\delta$, with $\delta>0$.
Then there exist infinitely many $n$ such that for all $k_1,\dots,k_n\in\alphabet$
\begin{align*}
  \onenorm{{\cM}(k_1) \cdots {\cM}(k_n)}^{1/n} \ge \lsrn{\matrixset,\onenormop} > 1+\delta/2.
\end{align*}
This implies that there exists an $n$ such that for all $k=\radix{k_1\dots k_n}\in\alphabet^{(n)}$
\begin{align*}
  \max\left(\gamma^{(n)}_{k+1}, \gamma^{(n)}_k\right)
   =  \onenorm{\cM^{(n)}_k}=\onenorm{{\cM}(k_1) \cdots {\cM}(k_n)} > \frac{1}{R},
\end{align*}
where by irreducibility we can take  the constant $R$ as in \eqref{eq:R def}.
Using \lemref{bounded skewness}, this implies that for this $n$ and all $k\in\alphabet^{(n)}$
\begin{align*}
     \gamma^{(n)}_k
 \ge \frac{\min(\gamn{k},\gamn{k+1})}{
           \max(\gamn{k},\gamn{k+1})} \cdot \max(\gamn{k},\gamn{k+1})
   > R \cdot \frac{1}{R} = 1,
\end{align*}
so  statement (a) follows also from applying \thmref{central}.
\end{proof}

We remark that the assumptions $\vc{p}=\vc{q}$ and irreducibility
 are used only in the second part of the proof,
 hence these conditions are not necessary for statement (b) in the theorem.

\subsection{Scope of the theorem}
We consider the following two questions:
\begin{itemize}
 \item When do we get $\lsr{\Sigma_\cM}=1$ or a reducible  $\Sigma_{\cM}$, the cases the theorem says nothing about?
 \item How can we calculate $\lsr{\Sigma_\cM}$? Is there an explicit expression or algorithm for calculating it?
\end{itemize}
There is good and there is bad news here:
 the good news is that the cases $\lsr{\Sigma_\cM}=1$ and $m_e=0$ for some $\abs{e}<M$
 happen only for a very limited set of vectors of marginal probabilities $\vc{p}$ in $[0,1]^M$,
 but the bad news is that the lower spectral radius is in general hard to calculate (see \cite{TB}).

First note that $m_e=0$ for some $\abs{e}<M$ happens only when at least one of the $p_i=0$, $i\in\alphabet$.
Thus
\begin{align*}
 E_0 := \set{\vc{p}\in[0,1]^M: \Sigma_\cM \,{\rm reducible} }
\end{align*}
 has dimension at most $M-1$, and hence has empty interior and Lebesgue measure zero.
The other exceptional set in the theorem is
\begin{align} \eqlbl{P_1 def}
 E_1 := \set{\vc{p}\in[0,1]^M: \lsr{\Sigma_\cM}=1}.
\end{align}
The lower spectral radius has the following \kw{scaling property} with respect to the vector of marginal probabilities:
 if $\tilde{\vc{p}}=c\vc{p}$ for some $c\ge0$, then $\tilde{\cM}(k)=c^2\cM(k)$ for all $k\in\alphabet$ and hence
  $\lsr{\Sigma_{\tilde{\cM}}}=c^2\lsr{\Sigma_\cM}$.
Thus for each vector of marginal probabilities $\vc{p}$ at most one scalar multiple is in $E_1$.
Similarly, if $\tilde{\vc{p}}\ge\vc{p}$ component-wise, then $\tilde{\cM}(k)\ge\cM(k)$ component-wise
 and hence $\lsr{\Sigma_{\tilde{\cM}}}\ge\lsr{\Sigma_\cM}$.
Combining this with the scaling property, it follows that if $\tilde{\vc{p}}>\vc{p}$ component-wise,
 then also $\lsr{\Sigma_{\tilde{\cM}}}>\lsr{\Sigma_\cM}$.
From the scaling property it also follows that $E_1$ has empty interior.

With respect to the irreducibility condition for  $\Sigma_{\cM}$ it is interesting to consider
 the second example given in \cite{DS}, Section 7.
Here  $\vc{p}=(1,0,p,0,1)$ with $p\in[0,1]$.
(See \figref{example p=(1,0,p,0,1)})
The corresponding (reducible) set of expectation matrices is given by
\begin{align*}
 (\exmat{k})_{k\in\alphabet} &=
 \left(
  \begin{bmatrix} 1     & 0  \\ 0  & 2+p^2 \end{bmatrix},
  \begin{bmatrix} 0     & 1  \\ 2p & 0     \end{bmatrix},
  \begin{bmatrix} 2p    & 0  \\ 0  & 2p    \end{bmatrix},
  \begin{bmatrix} 0     & 2p \\ 1  & 0     \end{bmatrix},
  \begin{bmatrix} 2+p^2 & 0  \\ 0  & 1     \end{bmatrix}
 \right),
\end{align*}
hence $(\gam{k})_{k\in\alphabet} = (2+p^2,1,2p,2p,1)$.

\begin{figure}[t]
\centering
\includegraphics*[width =11cm]{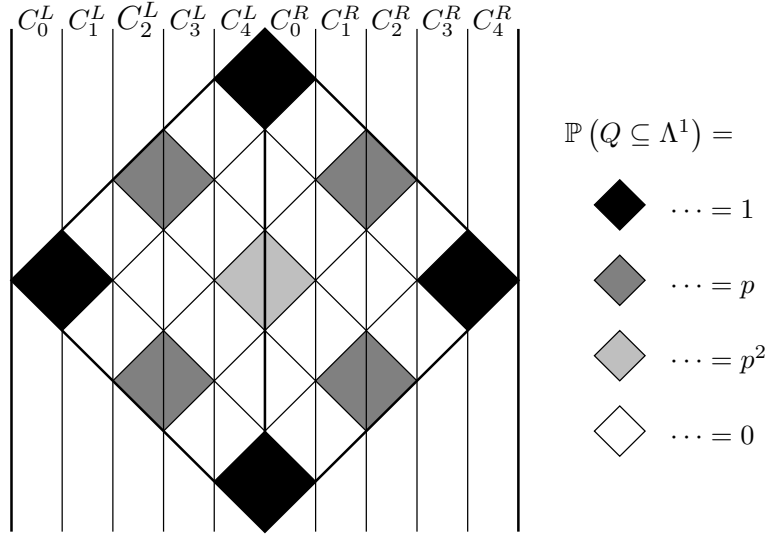}
\caption{The parameterized family $\vc{p}=(1,0,p,0,1)$, $p\in[0,1]$.}
\figlbl{example p=(1,0,p,0,1)}
\end{figure}

The level 1 expectation matrices have the property that each row and column contains at most one non-zero element.
We will call matrices having this property \kw{permutative matrices}.
Clerly any product of permutative matrices is permutative again.
Moreover, letting $\pi(A)$ denote the product of the non-zero entries of a permutative matrix $A$, it follows that
\begin{align*}
 \pi(\exmat{\tnk}) = \prod_{i=1}^{n} \pi(\exmat{k_i})
\end{align*}
for all $\tnk\in\tree$.
It is easy to see that for any permutative $D\times D$ matrix
\begin{align*}
 \infnorm{A}=\norm{A}_1\ge\sqrt[D]{\abs{\pi(A)}}.
\end{align*}
If  $p\ge\shalf$, then $\pi(\exmat{k})\ge2p$ for all $k\in\alphabet$.
Altogether, by plugging the two  equations above into the definition of the lower spectral radius,
 we find that
\begin{align*}
  \lsr{\Sigma_\cM} \ge  \sqrt{2p} >1, \quad \mathrm{if}\; p>\shalf.
  \end{align*}
However, in \cite{DS} it is  shown that  for \emph{all} $p<1$ the algebraic difference $F_1-F_2$ contains no interval a.s.
This  example thus shows that at least some irreducibility condition
 is necessary in \thmref{central with lsr}.

\section{Classifying 2-adic random Cantor sets} \seclbl{2-adic Cantor sets}
In this section we consider the symmetric case for $M=2$. For short we write
$(\alpha,\beta):=(p_0,p_1)=(q_0,q_1)$ for the marginals. Note that
$$p_0+p_1=\mu(\{0\})+\mu(\{1\})+2\mu(\{0,1\}).$$
 Since we require $p_0+p_1> 1$ (recall \eqref{eq:no extinction}), it follows that $\mu(\{0,1\})>0$, and so the  joint survival condition is always verified.

\subsection{Expectations}
The expectation matrices are given by
\begin{align} \eqlbl{2-adic exmat expr}
  \exmat{0} &= \begin{bmatrix}\alpha\beta & 0 \\ \alpha\beta & \alpha^2+\beta^2 \end{bmatrix}, &
  \exmat{1} &= \begin{bmatrix}\alpha^2+\beta^2 & \alpha\beta \\ 0 & \alpha\beta \end{bmatrix},
\end{align}
and the $\gam{k}$ and $m_e$ are given by
\begin{align*}
 \gam{0} = \alpha^2+\beta^2,\;  \gam{1} = 2\alpha\beta,\qquad m_0 = \alpha^2+\beta^2,\;  m_1 = \alpha\beta,\;  m_2 &= 0.
\end{align*}
Recursion relation \eqref{eq:one step gamma recursion with m_e} hence becomes
\begin{equation} \eqlbl{2-adic gamma recursion}
\begin{alignedat}{2}
   \gam[(n+1)]{2k}   &= m_0                            \gamn{k}
                     &&= \left(\alpha^2+\beta^2\right) \gamn{k}, \\
   \gam[(n+1)]{2k+1} &=  m_1         \left(\gamn{k}+\gamn{k+1}\right)
                     &&= \alpha\beta \left(\gamn{k}+\gamn{k+1}\right),
\end{alignedat}
\end{equation}
for all $n\ge0$ and $k\in\integers$.

Note that $p_0+p_1>1$ implies that $\alpha,\beta>0$. Therefore both $m_0>0$ and $m_1>0$.

\subsection{Neighbour bounds}
The skewness lower bound from \lemref{bounded skewness} is reduced to the simple expression
\begin{align} \eqlbl{2-adic R}
  R = \frac{m_1}{m_0} = \frac{\alpha\beta}{\alpha^2+\beta^2}.
\end{align}
The following lemma provides a bound that is even sharper than that of \lemref{bounded skewness}
 and shows a specific ordering of neighbouring \hogams{}
 that appears in the symmetric 2-adic algebraic difference.
See also \figref{2-adic neighbor bounds}.

 \begin{figure}[h]
\centering
\includegraphics*[width =11cm]{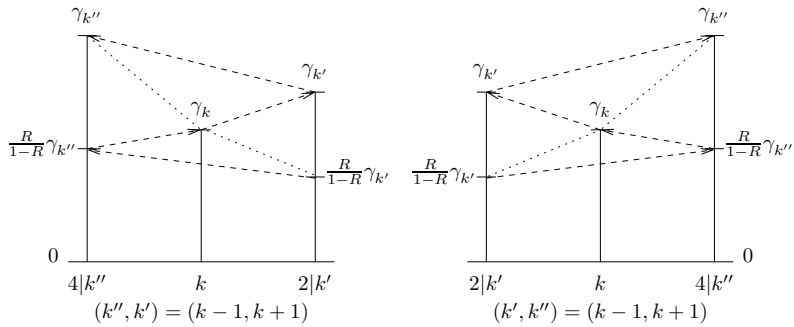}
\caption{A graphical representation of \lemref{2-adic gamma bounds}, with the $\gam{k}$ displayed as a bar graph.
 The situation is shown for the two solutions to the equation $\{k',k''\}=\{k-1,k+1\}$.
 The arrows indicate the order of the inequalities in the lemma.}%
\figlbl{2-adic neighbor bounds}%
\end{figure}

\begin{lem} \lemlbl{2-adic gamma bounds}
Consider the algebraic difference between two independent 2-adic random Cantor sets
 that have equal vectors of marginal probabilities.
Then
\begin{align} \eqlbl{2-adic gamma bounds}
      \frac{R}{1-R} \gamn{k'}
  \le \frac{R}{1-R} \gamn{k''}
  \le \gamn{k}
  \le \gamn{k'}
  \le \gamn{k''}
\end{align}
 for all $n\ge0$, odd $k\in\integers$ and $\{k',k''\}=\{k-1,k+1\}$ such that $4|k''$ and $2|k'$.
\end{lem}
\begin{proof}
First note that since $m_0=\gam{0}\ge\gam{1}=m_{-1}+m_1=2m_1$, we have $R\le\half$,
 so the fraction $\frac{R}{1-R}$ is always well defined.

The proof is by induction on $n$.
Since $\gam[(0)]{k}=1$ for all $k\in\integers$, the case $n=0$ is trivial.
Now assume that \eqref{eq:2-adic gamma bounds} holds for some $n\ge0$,
 and let $k,k',k''$ be as stated.
Note that $k'/2$ is odd and $k''/2$ even, and $\abs{k'/2-k''/2}=1$.
The first and the last inequality of \eqref{eq:2-adic gamma bounds} follow by
\begin{align} \eqlbl{2-adic even gamma ineq derivation}
                                                 \gam[(n+1)]{k'}
 \stackrel{\eqref{eq:2-adic gamma recursion}}{=} m_0 \gamn{k'/2}
 \stackrel{\eqref{eq:2-adic gamma bounds}}{\le}  m_0 \gamn{k''/2}
 \stackrel{\eqref{eq:2-adic gamma recursion}}{=} \gam[(n+1)]{k''}.
\end{align}
The middle two inequalities  of \eqref{eq:2-adic gamma bounds} follow by
\begin{align*}
                                                                                  \frac{R}{1-R} \gam[(n+1)]{k''}
 & \stackrel{\eqref{eq:2-adic R}}{=}               \frac{m_1}{m_0} \left(\frac{R}{1-R}+1\right) \gam[(n+1)]{k''}
  \stackrel{\eqref{eq:2-adic gamma recursion}}{=} m_1             \left(\frac{R}{1-R}+1\right) \gamn{k''/2} \\
 & \stackrel{\eqref{eq:2-adic gamma bounds}}{\le}  m_1             \left(\gamn{k'/2} + \gamn{k''/2}\right)
   \stackrel{\eqref{eq:2-adic gamma recursion}}{=}                                              \gam[(n+1)]{k}
  \stackrel{\eqref{eq:2-adic gamma bounds}}{\le}  m_1             \left(1+\frac{1-R}{R}\right) \gamn{k'/2} \\
 & \stackrel{\eqref{eq:2-adic gamma recursion}}{=} \frac{m_1}{m_0} \left(1+\frac{1-R}{R}\right) \gam[(n+1)]{k'}
   \stackrel{\eqref{eq:2-adic R}}{=}                                                            \gam[(n+1)]{k'},
\end{align*}
where in the last inequality we used \eqref{eq:2-adic gamma bounds}, multiplied by a factor $\frac{1-R}{R}$.
\end{proof}


\subsection{The smallest correlation coefficient}
\lemref{2-adic gamma bounds} allows us to pinpoint for each $n\ge0$
 a $k_n$ such that $\gamn{k_n}$ equals the minimal value
\begin{align*}
     \gamma^{(n)}
   = \min_{k\in\integers} \gamn{k}
   = \min_{k\in\alphabet^{(n)}} \gamn{k}.
\end{align*}
Obviously $k_n$ should be an odd number, but we can say more:
 one of the two neighbours of $2k_n$ will turn out to serve well as $k_{n+1}$.
Even though this selection procedure is very `local', it will still pinpoint a global minimum.

In order to decide which neighbour of $2k_n$ has to be chosen,
 an auxiliary sequence $\seq{k'_n}$ is defined.
This is done in such a way such that $k'_n$ is the neighbour of $k_n$
 that has the smallest value of $\gamn{k'_n}$.
Define the sequences $\seq{k_n}$ and $\seq{k'_n}$ by
\begin{align} \eqlbl{2-adic min seq defs}
  \begin{aligned}
    k_0  &:= 0, & k_{n+1}  &:= 2k_n + (k'_n-k_n) = k'_n+k_n, \\
    k'_0 &:= 1, & k'_{n+1} &:= 2k_n,
  \end{aligned}
\end{align}
for all $n\ge0$.
Note that this indeed makes $k'_n$ a neighbour of $k_n$ for all $n\ge0$,
 but their relative order alternates for subsequent $n$:
 $k'_{2n} = k_{2n}+1$, but $k'_{2n+1} = k_{2n+1}-1$ for all $n\ge0$.

We mention that the sequence $\seq{k_n}$ is also known as the Jakobsthal sequence,
 since $k_{n+2}=k_{n+1}+2k_n$ for $n\ge0$ and $k_0=0$, $k_1=1$.


\begin{lem} \lemlbl{2-adic min seq}
Let $\seq{k_n}_{n\ge0}$ and $\seq{k'_n}_{n\ge0}$ be as defined in \eqref{eq:2-adic min seq defs},
 then for all $n\ge0$
\begin{align} \eqlbl{2-adic min seq}
     \gamn{k_n}  &= \min_{k\in\integers} \gamn{k}
                      = \min_{k\in\integers} \gamn{2k+1}, &
     \gamn{k'_n} &= \min_{k\in\integers} \gamn{2k}.
\end{align}
\end{lem}
\begin{proof}
The proof is by induction on $n$.
For $n=0$ the statements are trivial because $\gam[(0)]{k}=1$ for all $k\in\integers$.
Now suppose that the statement of the lemma holds for a certain $n\ge0$.
Then, by using the induction hypothesis, the last equality of \eqref{eq:2-adic min seq} follows from
\begin{align} \eqlbl{2-adic min gamma aux expr}
         \min_{k\in\integers} \gam[(n+1)]{2k}
  &= m_0 \min_{k\in\integers} \gamn{k}
   = m_0 \gamn{k_n}
   = \gam[(n+1)]{2k_n}
   = \gam[(n+1)]{k'_{n+1}}.
\end{align}
Note that by the induction hypothesis
 \[ \gamn{k_n} =\min_{k\in\integers}\gamn{2k+1} \quad \text{and} \quad
    \gamn{k'_n}=\min_{k\in\integers}\gamn{2k}, \]
thus by observing that of any two consecutive integers one is always even and one is always odd,
  it follows that
 \[ \min_{k\in\integers} \left(\gamn{k}+\gamn{k+1}\right) = \gamn{k_n} + \gamn{k'_n}. \]
By applying \lemref{2-adic gamma bounds} the first, double equality from \eqref{eq:2-adic min seq} follows:
\begin{align}
     \min_{k\in\integers} \gam[(n+1)]{k}
   = \min_{k\in\integers} \gam[(n+1)]{2k+1}
  &= m_1 \min_{k\in\integers} \left(\gamn{k} + \gamn{k+1}\right) \notag\\
  &= m_1 \left(\gamn{k_n} + \gamn{k'_n}\right)
   = \gam[(n+1)]{2k_n+(k'_n-k_n)}
   = \gam[(n+1)]{k_{n+1}}. \eqlbl{2-adic min gamma expr}
\end{align}
\end{proof}

\subsection{Limit behavior of \texorpdfstring{$\gamma^{(n)}$}{gamma\textasciicircum(n)}}
Define the sequence $\seq{a_n}_{n\ge0}$ of minimum values by setting
\begin{align} \eqlbl{a_n def}
  a_n := \gamn{k_n},
\end{align}
 for all $n\ge0$.
Using \eqref{eq:2-adic min gamma expr} and \eqref{eq:2-adic min gamma aux expr}
 the following recurrence relation is obtained:
\begin{align} \eqlbl{a_n recurrence relation}
    a_{n+2}
  = \gam[(n+2)]{k_{n+2}}
 &= m_1 \left(\gam[(n+1)]{k_{n+1}} +     \gam[(n+1)]{k'_{n+1}}\right) \notag\\
 &= m_1 \left(\gam[(n+1)]{k_{n+1}} + m_0 \gamn{k_n}       \right)
  = m_1 a_{n+1} + m_1m_0 a_n,
\end{align}
for all $n\ge0$. The initial conditions are given by
\begin{align} \eqlbl{a_n initial conditions}
  a_0 &= \gam[(0)]{k_0} = 1, &
  a_1 &= \gam[(1)]{k_1} = \gam{1} = 2m_1.
\end{align}
The characteristic polynomial of this linear recurrence relation is $h(x)=x^2-m_1x-m_0m_1$
with roots
\begin{equation*}
  x_{\pm} = \shalf m_1 \pm \sqrt{m_1 m_0 + \sfrac{1}{4} m_1^2}.
\end{equation*}

Since $m_1>0$,
 $x_+$ and $x_-$ are two distinct roots of the characteristic equation,
  so  the general form of the solution to the recurrence relation is given by
\begin{align*}
 a_n = c_+^{} x_+^n + c_-^{} x_-^n,
\end{align*}
 for all $n\ge0$, where the constants $c_+$ and $c_-$ are determined by the initial conditions.
Since $x_+>0$ is the largest zero of the parabola $h$ we have
\begin{equation*}
  \abs{x_+}<1 \quad\Leftrightarrow\quad h(1)>0  \quad\Leftrightarrow\quad m_1(m_0+1) < 1,
\end{equation*}
thus we can conclude that
\begin{align} \eqlbl{a_n lim}
   \lim_{n\to\infty} \gamn{k_n}
 = \lim_{n\to\infty} a_n &= \begin{cases}
     0,      & \abs{x_+} < 1, \\
          \infty, & \abs{x_+} > 1,
   \end{cases} = \begin{cases}
     0,      & m_1(m_0+1) < 1, \\
         \infty, & m_1(m_0+1) > 1.
   \end{cases}
\end{align}
Using the rightmost equations in \eqref{eq:2-adic min gamma aux expr},
 we find that the neighbour sequence $\gamn{k'_n}$ behaves in exactly the same way.

Altogether this leads to the following result:

\begin{thm} \thmlbl{2-adic central}
Consider the algebraic difference $F_1-F_2$ between two independent 2-adic random Cantor sets $F_1$ and $F_2$
 whose joint survival distributions have the same marginal probability vectors.
 \begin{itemize}
 \item If $C>1$, then $F_1-F_2$ contains an interval a.s.\ on $\set{F_1-F_2\neq\es}$.
 \item If $C<1$, then $F_1-F_2$ contains no interval a.s.
\end{itemize}
Here the  number $C$ is defined by
\begin{align}
  C := m_1(1+m_0) = \alpha\beta(1+\alpha^2+\beta^2) = p_0p_1(1+p_0^2+p_1^2).
\end{align}
\end{thm} 
\begin{proof}
We will show that the  conditions for \thmref{central}
 do hold for Cantor sets of appropriate higher order. We already remarked that the JSC holds for 2-adic
 Cantor sets, and so it holds for all higher order Cantor sets.

If $m_1(1+m_0)<1$, then $\gamn{k_n}\to0$ and $\gamn{k'_n}\to0$ as $n\to\infty$,
 so for some $n$ large enough, $\gamn{k_n}$ and $\gamn{k'_n}$ are both strictly smaller than $1$.

If $m_1(1+m_0)>1$, then $\gamn{k_n}\to\infty$ as $n\to\infty$,
 so for some $n$ large enough, $\gamn{k_n}$ is strictly larger than $1$,
 and by \lemref{2-adic min seq} this holds for \emph{all} $\gamn{k}$.
\end{proof}

%

\begin{remark}
The number $C$ in \thmref{2-adic central} is \emph{not} the spectral radius of the
collection $$\Sigma_\cM=\{\exmat{0} ,\exmat{1}\}.$$
In fact, let $\lambda_{PF}(A)$ denote the Perron-Frobenius eigenvalue of a matrix $A$.
It is well known that $\lambda_{PF}(A)$ is sandwiched between the smallest and the largest
 column sum of $A$. Combining this with \lemref{bounded skewness} we obtain
$$\gamn{k_n}\!=\!\min_{k\in\integers}\min\{\gamn{k}\!, \gamn{k+1}\}\le \min_{k\in\integers}\lambda_{PF}(\cM^{(n)}{k} )\le
    \frac1{R}\min_{k\in\integers}\min\{\gamn{k}\!, \gamn{k+1}\}\!=\!\frac1{R}\gamn{k_n}.$$
    According to Theorem B.1 of \cite{Gurvits}, $(\min_{k\in\integers}\lambda_{PF}(\exmat{k} ))^{1/n}$,
which is nothing else than the $n^{\mathrm{th}}$ root of the smallest spectral radius of all length $n$ products of matrices from $\Sigma$,
converges to $\lsr{\Sigma_\cM}$. It follows with our results above that the lower spectral radius of $\Sigma_\cM=\{\exmat{0} ,\exmat{1}\}$
is thus equal to
$$\lim_{n\rightarrow\infty} a_n^{1/n}=\lim_{n\rightarrow\infty} (c_+^{} x_+^n + c_-^{} x_-^n)^{1/n}=x_+=\shalf m_1 + \sqrt{m_1 m_0 + \sfrac{1}{4} m_1^2}.$$
This might be of independent interest: we have shown that for $\max\{2b,1-2b\}\le a\le 1+b^2$ the lower spectral radius of the collection consisting of
\begin{align} \eqlbl{two matrices}
  \exmat{0} &= \begin{bmatrix}a & b \\ 0 & b \end{bmatrix}, \quad \exmat{1} = \begin{bmatrix}b & 0\\ b & a \end{bmatrix},
\end{align}
is equal to $\shalf b + \shalf\sqrt{4ab +  b^2}$.
\end{remark}

\figref{symmetric 2-adic boundaries} gives an overview of boundaries
 in the space of vectors of marginal probabilities $\vc{p}$
 that separate areas  where different sets of conditions  imply the absence or presence of intervals.
The figure also indicates the area where the Palis conjecture fails,
 i.e., the area where \eqref{eq:Palis conjecture cond} does not imply that $F_1-F_2$ contains an interval
 (on $\set{F_1-F_2\neq\es}$).

\begin{figure}[h!]
\centering
\includegraphics*[width =8cm]{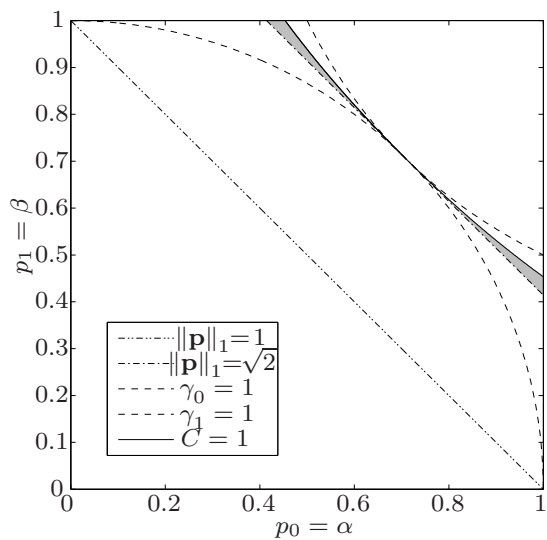}
\vspace*{-0.4cm}\caption{%
Classification of the $2$-adic symmetric algebraic difference in the $(p_0,p_1)$ plane:
 (1) Below $\norm{\vc{p}}_1=1$  a.s~ $F_1=F_2=\emptyset$.
 (2) Below $\norm{\vc{p}}_1=\sqrt{2}$ there are no intervals because $\Hdim(F_1-F_2)<1$.
 (3) Below $\gam{0}=1$ the `no intervals' part of \protect\thmref{central} holds.
 (4) Above $\gam{1}=1$ the `intervals' part holds.
 (5) The line $C=1$ is the separating boundary of \protect\thmref{2-adic central}.
The area where the Palis conjecture fails is shaded grey.}%
\figlbl{symmetric 2-adic boundaries}%
\end{figure}


\section{A proof for the basic result} \seclbl{general proof}

In the following we will give a proof for part (a) of \thmref{central} (we already mentioned that the proof of part (b)
is much simpler, and follows closely the proof in \cite{DS}).
The proof of \thmref{central} is based on the following observations.
The process of $n$-th level $M$-adic squares that are surviving in the level $n$ approximations $\Lambda^n$
 inherits the self-similarity property of the individual random Cantor sets $F_1$ and $F_2$:
 conditional on the survival of an $n$-th level $M$-adic square $Q_{\tn,\tnj}$,
 the (scaled) process starting at this surviving square has the same distribution
 as the whole process, which starts at $[0,1]^2$.
Moreover, conditional on the survival of a set of $n$-th level $M$-adic squares that are pairwise \emph{un}aligned,
 the processes in each of these squares are independent.

The columns behave very inhomogeneously:
  for every $n$ there are  columns which contain at most \emph{one} triangle. This observation led to the
 idea in \cite{DS}  to pair unaligned left and right triangles
 that survive in the \emph{same} column into what are called \kw{$\Delta$-pairs}.
The main idea  of the proof  for part (a)
 is to show that with positive probability a $\Delta$-pair will occur in some column $C$ of some $\Lambda^m$ (\lemref{Delta-pair wpp}) ,
 and that conditional on this, the pairs in all subcolumns will grow exponentially,
 so the projection of the $\Delta$-pairs within $C\cap\Lambda$ will be an $M$-adic interval (\lemref{exponential growth}).
 We will follow the structure of the proof in  \cite{DS}, each lemma there corresponds to a section with lemma here.

\subsection{Joint growth of $\Delta$-pairs}

With the `$\tnks$-th subcolumn of a level $m$ $\Delta$-pair' $(L^m,R^m)$
 that is contained in a column $C={C^U_{\tmls}}$ we will indicate
\begin{align*}
  C^U_{\tmls\tnks} \intersec (L^m \cup R^m),
\end{align*}
the intersection of the $\Delta$-pair with ${C^U_{\tmls\tnks}}$, the $\tnks$-th subcolumn of ${C^U_{\tmls}}$.

For such a $\Delta$-pair $(L^m,R^m)$,
 the distribution of the number of level $m+n$ $V$-triangles
 surviving in $\Lambda^{m+n}$
 in the $\tnks$-th subcolumn of $(L^m,R^m)$,
 conditional on the survival of $(L^m,R^m)$ in $\Lambda^m$,
 is independent of $m$, the particular choice of the column $C$ and the $\Delta$-pair in it.
Therefore, we can unambiguously denote a random variable having this distribution by
\begin{align} \eqlbl{tilde Z}
  \tilde{Z}^{V}(\tnks)
\end{align}
for all $V\in\{L,R\}$ and $\tnks\in\tree$.

 \begin{figure}[t]
\centering
\includegraphics*[width =11cm]{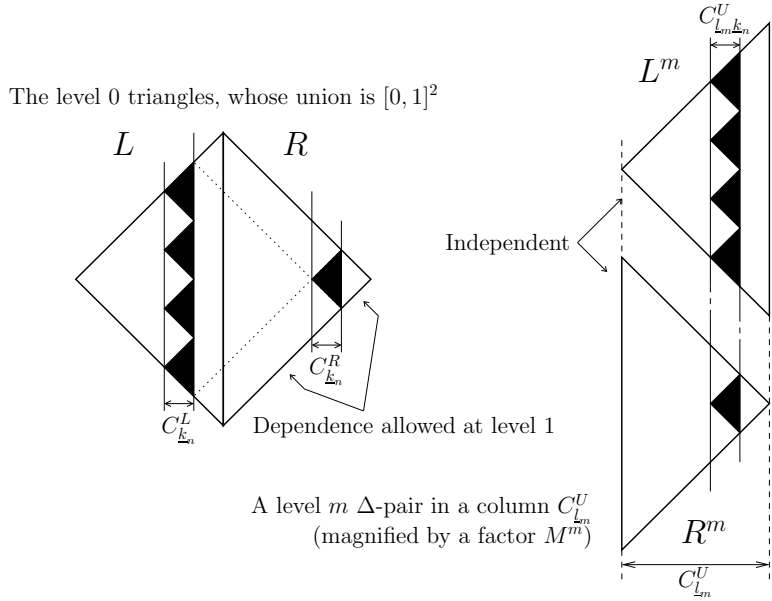}
\caption{The distinction between level $0$ triangles and a $\Delta$-pair.}%
\figlbl{level 0 pairs vs}%
\end{figure}

In general $\tilde{Z}^{V}(\tnks)$ does not have the distribution of $Z^{V}(\tnks)$
 because there \emph{is} possible dependence between the offspring generation of the two level $0$ triangles,
 whereas there is \emph{no} dependence between the offspring generation of the $L$-triangle and the $R$-triangle of a level $m$ $\Delta$-pair
 by unalignedness.
Essentially, $\tilde{Z}^{V}(\tnks)$ has the distribution of
 the sum of \emph{independent} copies of the random variables $Z^{RV}(\tnks)$ and $Z^{LV}(\tnks)$,
 but $Z^{V}(\tnks)$ is just the sum of these random variables, which may be dependent.
(Compare the level $0$ triangles with the $\Delta$-pair in \figref{level 0 pairs vs}.)
Clearly,  the $Z^{UV}(\tnks)$ are dependent if they count triangles that are aligned, e.g.,
$Z^{RR}(\tnks)$ and $Z^{LL}(\tnks)$ will in general not be independent. However, $Z^{RV}(\tnks)$ and $Z^{LV}(\tnks)$
will never be aligned.
But there can also be dependence between the offspring of  level $0$ squares
 --- if that is dictated by the joint survival distributions ---
 which induces dependence between the random variables $Z^{UV}(\tnks)$, at any level $n$.
See  \figref{level 0 pairs vs}.
Despite these differences, both do have the same expected value (by linearity of expectations), and this is all that is needed in the proof.

Let
\begin{align} \eqlbl{N(k_n) def}
 \tilde{N}(\tnks) = \min\{\tilde{Z}^{L}(\tnks),\tilde{Z}^{R}(\tnks)\}
\end{align}
for all $\tnks\in\tree$.
This is the distribution of the minimum number of triangles of each triangle type
 that survive in the $\tnks$-th subcolumn of a  $\Delta$-pair.
The next lemma states that with positive probability
 the growth of $\tilde{N}(\tnks)$ for all $\tnks$ up to a certain level $n$ is exponential.

\begin{lem} (Extension of Lemma 1 in \cite{DS})\lemlbl{finite exponential growth}
If $\gamma>1$, and the joint survival distribution(s) satisfy the joint survival condition, then for all $n\ge0$
\begin{align*}
 \prob{\tilde{N}(\tmks) \ge \gamma^m \text{ for all } \tmks\in\tree[m] \text{ for all } 0\le m \le n} > 0.
\end{align*}
\end{lem}
\begin{proof}
The  proof is  similar to the proof of Lemma 1 in \cite{DS},
 but it shows where and how the JSC emerges for general survival distributions.

For a joint survival distribution $\mu$
 we uniquely define the joint survival distribution $\mu^*$ by requiring that $\mu^*(\msupp{\mu}) = 1$.
(See \eqref{eq:msupp def} for the definition of the marginal support $\msupp{\mu}$.)
For example, if $\vc{p}=(1,\tfrac{1}{\pi},0,\shalf)$,
 then $\msupp{\mu}=\set{0,1,3}$ and $\vc{p}^*=(1,1,0,1)$.
Note that $\msupp{\mu} = \msupp{\mu^*}$.
We obtain $\mu^*,\lambda^*$ from $\mu, \lambda$ in this way.
We mark all entities that refer to the substitution of   $\mu,\lambda$ by $\mu^*,\lambda^*$ with a $^*$ superscript.
Note that we have $\cM^*(\tnks)\ge\cM(\tnks)$ component-wise for all $\tnks\in\cT$.

Consider for each $n\ge0$ the event
\begin{align*}
  J_n &:=  \Intersec_{m=0}^n \bigcup\set{Q_{\tms,\tmjs} \subseteq \Lambda^m: \tms,\tmjs \in \tree[m]\;with\; p_\tms>0,p_\tmjs>0}.
\end{align*}
This is the event that in the first $n$ steps of the construction,
 all squares that have positive marginal survival probability do survive all jointly.
Let $a$ be the \emph{number} of level 1 squares having positive marginal survival probability.
For each $n\ge0$ the number of such squares at level $n$ is given by $a^n$
 and the event $J_n$ has probability $\prob{J_n} = \prob{J_1}^{1+a+\dots+a^{n-1}}$where $\prob{J_1}=\mu(\msupp{\mu})\lambda(\msupp{\lambda})$.
At this point we use the joint survival condition, since it implies that $\prob{J_1}>0$ and thus $\prob{J_n}>0$.
Note that by construction we have $\probstar{J_n}=1$.

Let $\tilde{N}(\tmks)= \min\{\tilde{Z}^{L}(\tmks),\tilde{Z}^{R}(\tmks)\} $.
By the self-similarity of the process and the requirement that
 the process runs independently in the triangles of a $\Delta$-pair,
 the event that in the first $n\ge0$ sublevels of the surviving $\Delta$-pair
 all triangles (in the $\Delta$-pair) that have positive probability to survive, do survive simultaneously,
 occurs with at least probability $(\prob{J_n})^2>0$.
Conditional on this latter event --- which has positive probability ---
 the following chain of component-wise (in)equalities holds:
\begin{align*}
&\se \begin{bmatrix}     \tilde{Z}^{L}   (\tmks) \;\;     \tilde{Z}^{R}   (\tmks) \end{bmatrix}
   = \begin{bmatrix}    {\tilde{Z}^{*;L}}(\tmks) \;\;     {\tilde{Z}^{*;R}}(\tmks) \end{bmatrix} \notag\\
  &= \begin{bmatrix} \ex{\tilde{Z}^{*;L}}(\tmks)\;\;   \ex{\tilde{Z}^{*;R}}(\tmks) \end{bmatrix}
   = \begin{bmatrix} \ex{      {Z}^{*;L}}(\tmks)\;\;   \ex{      {Z}^{*;R}}(\tmks) \end{bmatrix} \notag\\
  &= \vc{e} \cM^*(\tmks)
 \ge \vc{e} \cM  (\tmks)
 \ge \gamma \vc{e} \cM(k_2) \cdots \cM(k_m)\ge\dots
 \ge \gamma^m \vc{e}, \eqlbl{exmat col sum gamma growth}
\end{align*}
 for all $0\le m\le n$ and $\st{k}_m\in\cT_m$, and where $\vc{e}$ is the $2$-dimensional all-one row vector.
This directly implies the statement of the lemma.
\end{proof}

\subsection{Existence of a $\Delta$-pair}

\begin{lem} (Extension of Lemma 2 in \cite{DS})\lemlbl{Delta-pair wpp}
If $\gamma>1$, then
\begin{align*}
  p_{\Delta} := \prob{\exists m\ge0 \text{ s.t.\ there exists a level }m\text{ $\Delta$-pair in }\Lambda^m} > 0.
\end{align*}
\end{lem}

\begin{proof}
We will show that one can take $m=2$:  if $\gamma>1$, then
\begin{align*}
  \prob{\text{there exists a level }2\text{ $\Delta$-pair in }C^{R}_{00}\cap \Lambda^2} > 0.
\end{align*}
The corresponding geometric structure is  depicted in \figref{projection}.

From $\gam{0}=\sum_{i=0}^{M-1} \prob{Q_{i,i} \subseteq \Lambda_1}>1$ we may conclude
  that there exist distinct $a,b\in\{0,\dots,M-1\}$ such that
 $$p^{ab}:=\prob{Q_{a,a}, Q_{b,b} \subseteq \Lambda^1}>0.$$

From $\gam{1}=\prob{Q_{0,M-1} \subseteq \Lambda^1} + \sum_{i=1}^{M-1} \prob{Q_{i,i-1} \subseteq \Lambda^1}>1$
 we may conclude that there exists at least one $c\in\{1,\dots,M-1\}$ such that
 $p^{c}:=\prob{Q_{c,c-1} \subseteq \Lambda^1}>0$.

Using the independence of the process in the unaligned squares $Q_{a,a}$ and $Q_{b,b}$ and
 the self-similarity of the process, we now conclude that
\begin{align*}
 &\se\prob{ Q_{aa,aa},Q_{bc,b(c-1)}\subseteq\Lambda^2} \\
 & = \cprob{Q_{aa,aa},Q_{bc,b(c-1)}\subseteq\Lambda^2}{Q_{a,a},Q_{b,b}\subseteq\Lambda^1}
                                                 \prob{Q_{a,a},Q_{b,b}\subseteq\Lambda^1} \\
 & = \cprob{Q_{aa,aa}              \subseteq\Lambda^2}{Q_{a,a},Q_{b,b}\subseteq\Lambda^1}
     \cprob{          Q_{bc,b(c-1)}\subseteq\Lambda^2}{Q_{a,a},Q_{b,b}\subseteq\Lambda^1} p^{ab} \\
 & = \prob{Q_{a,a}   \subseteq\Lambda^1}
     \prob{Q_{c,c-1} \subseteq\Lambda^1} p^{ab}
 \ge p^{ab}p^{c}p^{ab} > 0,
\end{align*}
which finishes the proof, as $(L_{bc,b(c-1)},R_{aa,aa})$ is a level $2$ $\Delta$-pair in $C^{R}_{00}$.
\end{proof}

\subsection{Unaligned triangles}
\seclbl{unaligned}

The following lemma is purely combinatorial, and serves to obtain independence between the triangles in a $\Delta$-pair.

\begin{lem}(Lemma 3 in \cite{DS}) \lemlbl{grouping}
We are given $N$ distinct odd numbers $o_1, \dots, o_N$ and $N$
distinct even numbers $e_1, \dots , e_N$. Then we can couple the
odd numbers with the even numbers and we can color the $N$ couples
with three colors (say $\mathtt{r},\mathtt{g}$ and $\mathtt{b}$)
such that no two numbers in pairs of the same color are adjacent
and all colors are used for at least $\lfloor N/3\rfloor $ pairs.
That is, there exists a permutation $\pi $ of $\left\{1,\dots
,N\right\}$ such that we can color the pairs
$$
(e_1,o_{\pi (1)}),\dots ,(e_N,o_{\pi (N)})
$$
with the three colors such that with each color we painted at least
$\lfloor N/3\rfloor$ pairs and
 for any (also if $\ell=k$)
$(e_k,o_{\pi (k)})$ and $(e_{\ell},o_{\pi (\ell)})$ having the
same color it is true that:
\begin{equation*}
|e_{\ell}-o_{\pi (k)}|>1.
\end{equation*}
\end{lem}

\subsection{Exponential growth}
\seclbl{a general proof}

\begin{lem}(Adaptation of Lemma 4 of \cite{DS}) \lemlbl{exponential growth}

If $\gamma>1$, then there exists an $1<\eta<\gamma$ such that
\begin{align*}
 p_I := \prob{\tilde{N}(\tnks)\ge\eta^n \text{ for all } \tnks\in\tree[n] \text{ for all } n\ge0} > 0.
\end{align*}
\end{lem}
\begin{proof}
We can follow literally the proof of Lemma 4 in \cite{DS}, except that we define here the sets
\begin{align} \eqlbl{A_n def}
  A_n := \set{\tilde{N}(\tmks) \ge \eta^m \text{ for all } \tmks\in\tree[m] \text{ for all } 0\le m \le n},
  \end{align}
  instead of $A_n = \set{\tilde{N}(\tmks) \ge \eta^m \text{ for all } \tmks\in\tree[m]}$. The (rather embarrassing)
  reason is that the equality $\mathbb{P}\left(A_{n+1}^{c}|A_r\cap \dots \cap
A_n\right)=\mathbb{P}\left(A_{n+1}^c|A_n\right)$ on the bottom of page 10 in \cite{DS} is wrong in general.
All one needs is that the equality sign can be replaced by a $\le$ sign, but since proving
this seems to be rather involved (although intuitively obvious) we chose to redefine $A_n$ as in
\eqref{eq:A_n def}, since this conveniently leads to a sequence of decreasing  sets with intersection the set in the statement of
 \lemref{exponential growth}.
The proof then continues as in \cite{DS}, deducing from \lemref{finite exponential growth} and \lemref{grouping}   that
\begin{align*}
      p_I=\probBig{\Intersec_{n\ge 0} A_n}
    = \lim_{n\rightarrow\infty} \prob{A_n}  > 0,
\end{align*}
which finishes the proof.
\end{proof}

Lemma \ref{lem:exponential growth} ensures that with positive probability $p_{I}$
 the offspring in all subcolumns of a surviving $\Delta$-pair never dies out.
Lemma \ref{lem:Delta-pair wpp} ensures that with positive probability $p_{\Delta}$ a surviving $\Delta$-pairs exists.
Using the self-similarity of the process, we thus have the following corollary:

\begin{cor} \corlbl{interval in square projection wpp}
If $\gamma>1$, then for all $\tns,\tnjs \in\tree$
\begin{align*}
 \cprob{\phi(\Lambda \intersec Q_{\tns,\tnjs} ) \text{ contains an interval }}{
        Q_{\tns,\tnjs} \subseteq \Lambda^n} \ge p_{\Delta} p_{I} > 0.
\end{align*}
\end{cor}

As the next step, it is shown in \cite{DS}  that when $\Lambda\neq\es$,
 the maximum number of unaligned surviving squares at level $n$ grows to infinity as $n\to\infty$.
In each  unaligned surviving square
 the process runs independently and identically distributed to the process starting at $[0,1]^2$.
Because the number of these squares grows arbitrarily large,
 and in each square there is a positive probability that its projection contains a non-empty interval,
 this implies that almost surely the projection of $\Lambda$ contains an interval.

\nocite{MR1256400}
\bibliographystyle{amsalpha}
\bibliography{Cantor-refs}

\end{document}